\documentclass[11pt, reqno]{amsart}
\usepackage{amssymb,amsmath,epsfig,mathrsfs, enumerate, xparse, mathtools}
\usepackage[pagewise]{lineno}
\allowdisplaybreaks
\usepackage{graphicx}\usepackage[normalem]{ulem}
\usepackage{fancyhdr}
\usepackage[margin=2.5cm]{geometry}
\usepackage{hyperref}
\hypersetup{
 colorlinks   = true,
 urlcolor     = blue,
 linkcolor    = blue,
 citecolor   = red ,
 bookmarksopen=true
}
 \usepackage[usenames]{color}

\theoremstyle{plain}

\usepackage{pdfpages}
\usepackage[framemethod=TikZ]{mdframed}

\mdfdefinestyle{JimFrame}{%
    linecolor=red,
    outerlinewidth=2pt,
    roundcorner=20pt,
    innertopmargin=\baselineskip,
    innerbottommargin=\baselineskip,
    innerrightmargin=25pt,
    innerleftmargin=25pt,
    backgroundcolor=white!50!white}

\newcommand{\La}{\mathcal{L}}
\newcommand{\Ln}{\langle}
\newcommand{\Rn}{\rangle}
\theoremstyle{plain}
   \newtheorem{theorem}[subsection]{Theorem}

   \newtheorem{remark}[subsection]{Remark}
   
   \newtheorem{definition}[subsection]{Definition}
   
   \newtheorem{example}[subsection]{Example}
   \newtheorem{lemma}[subsection]{Lemma}
   \newtheorem{corollary}[subsection]{Corollary}

\begin{document}

\author{Abhishek Ghosh}
\address{Abhishek Ghosh
\endgraf TIFR-Centre of Applicable Mathematics, Bangalore-560065, India
 \endgraf And
\endgraf Department of Mathematics: Analysis, Logic and Discrete Mathematics
	\endgraf Ghent University
	\endgraf Krijgslaan 281, Building S8,	B 9000 Ghent,
	Belgium.}
\email{abhi21@tifrbng.res.in; abhi170791@gmail.com}

\author{Michael Ruzhansky}
\address{Michael Ruzhansky
\endgraf Department of Mathematics: Analysis, Logic and Discrete Mathematics
	\endgraf Ghent University
	\endgraf Krijgslaan 281, Building S8,	B 9000 Ghent,
	Belgium\\
 \endgraf And
 \endgraf School of Mathematical Sciences
\endgraf Queen Mary University of London
\endgraf United Kingdom.}
\email{michael.ruzhansky@ugent.be}


\thanks{The second author is supported by the FWO Odysseus 1 grant G.0H94.18N: Analysis
and Partial Differential Equations, the Methusalem programme of the Ghent University Special Research Fund (BOF) (Grant number 01M01021). MR is also supported
by EPSRC grant EP/R003025/2 and FWO Senior Research Grant G01152}

%\date{14 August, 2010}

\title[oscillating multipliers on stratified groups]{Sparse bounds for oscillating multipliers on stratified groups}

\maketitle

\begin{abstract} 
In this article, we address sparse bounds for a class of spectral multipliers that include oscillating multipliers on stratified Lie groups. Our results can be applied to obtain weighted bounds for general Riesz means and for solutions of dispersive equations.
\end{abstract}

\section{Introduction}
On $\mathbb{R}^n,$ operators of the form $\widehat{Tf}(\xi)=m_{\theta, \beta}(\xi) \widehat{f}(\xi),$ where $m_{\theta, \beta}=|\xi|^{-\frac{\theta\beta}{2}} e^{i|\xi|^{\theta}}\chi_{\{|\xi|>1\}}$ are known as oscillating multipliers. They are extensively studied starting with the pioneering works of Hardy, Hirschman \cite{Hirsch} and Wainger \cite{W}. Charles Fefferman proved the crucial weak type $(1, 1)$ estimates in \cite{F}, and  the sharp range for $L^p$ estimates was obtained by Fefferman and Stein in \cite{FS1}. They were also studied by H\"ormander \cite{H}. We also refer to the articles \cite{M, M1, P} for results in the context of the wave operators, that is, $\theta=1$. Weighted estimates for oscillating multipliers on $\mathbb{R}^n$ were initiated by Chanillo \cite{Chan} and extended in \cite{CKW1}. In \cite{CKW2}, weighted end-point estimates were obtained by Chanillo, Kurtz, and Sampson. In this article, we shall confine ourselves to weighted $L^p$ estimates for these operators on stratified Lie groups. In order to do that, let us recall the following preliminaries.   

Let ${\mathfrak g}$ be a $d$-dimensional, graded nilpotent Lie algebra so that 
$$
{\mathfrak g} \ = \ \bigoplus\limits_{i=1}^s \, {\mathfrak g}_i  
$$
as a vector space and $[{\mathfrak g}_i, {\mathfrak g}_j] \subset {\mathfrak g}_{i+j}$ for
all $i,j$. Suppose that ${\mathfrak g}_1$ generates ${\mathfrak g}$ as a Lie algebra. The associated, connected, simply connected Lie group $G$ is called a stratified Lie group. The homogeneous dimension of $G$ is defined as $Q \ = \ \sum_j j \, {\rm dim}({\mathfrak g}_j).$ Consider the sublaplacian ${\mathcal L} = - \sum_k X_k^2$ on $G$, where $\{X_k\}$ is a basis for ${\mathfrak g}_1$. For
any Borel measurable function $m$ on ${\mathbb R}_{+} = [0,\infty)$, 
we can define the spectral multiplier operator
$$
m(\sqrt{{\mathcal L}}) \ = \ \int_0^{\infty} m(\lambda) \, d E_{\lambda}
$$
where $\{E_{\lambda}\}_{\lambda\ge 0}$ is the spectral resolution of $\sqrt{{\mathcal L}}$. Since the exponential map is a global diffeomorphism, the measure on $G$ can be identified with the $d$-dimensional Lebesgue measure. In this setting, analogue of the classical H\"ormander-Mikhlin multiplier theorem was established in the seminal work by Christ in \cite{C}, also fundamental end-point estimates were obtained by Mauceri--Meda in \cite{MM}, see also \cite{MSeeger, MS, MS1} for other influential works.  In recent times, there are many important works in this context, we refer \cite{Bramati, Bui, Bui1, BuiIMRN, Bui2, CW, alessio-joint, plms, gafa, sikora}. We are inspired by the recent work \cite{CW} where the authors have introduced a general class of multipliers covering oscillating multipliers and obtained important end-point estimates. On more general graded groups Fourier multiplier operators are studied in \cite{Du1, Du2, FR, FR1,HR} and references therein. Throughout this article, for any Borel measurable set $R$ and $1\leq p<\infty,$ $\Ln f \Rn_{p, R}$ denotes $(\frac{1}{|R|}\int_{R}|f|^p)^{1/p}.$ Also, a family of sets $\mathcal{S}$ is called $\eta-$sparse if for each $R\in \mathcal{S}$ there exists $E_{R}\subset R$ such that $|E_R|\geq \eta |R|$ and $\{E_R\}_{\mathcal{S}}$ are pairwise disjoint. Now we state our main result.

\subsection*{Statement of main results}
Motivated by \cite{CW}, we introduce the following class of multipliers. Let $\nu>1$ be a number which will be specified later. Let $\phi$ be a smooth function on $(0, \infty),$ supported
on $\{\nu^{-1}\leq \lambda\leq \nu\}$ and satisfying $\sum_{j} \phi(\nu^{-j}\lambda)=1$ for all $\lambda>0.$ Define $m^j(\lambda) := m(\nu^j \lambda) \phi(\lambda).$
\begin{definition}
Let $\theta\in \mathbb{R}\setminus\{0\}$ and $\beta\geq 0.$ We say $m\in \mathscr{M}(\theta, \beta)$ if $m$ is supported in the set $\{\lambda\in \mathbb R_{+}: \lambda^{\theta}\geq 1\}$ and 
\begin{align}
&\sup_{j \theta>0} \nu^{j\theta\beta/2}\|m^{j}\|_{L^{\infty}(\mathbb R_{+})}<\infty,\\
\text{and}\ \ &\sup_{j \theta>0} \nu^{-j\theta(2s-\beta)/2}\|m^{j}\|_{L^2_{s}(\mathbb R_{+})}<\infty\ \ \ \text{for all}\ \  s\in \mathbb{N}. 
\end{align}
\end{definition}
\begin{example}
Let $\theta\in \mathbb{R}\setminus \{0\}.$ Define
$ m_{\theta, \beta}(\lambda):=e^{i\lambda^{\theta}}{\lambda^{-\frac{\theta \beta}{2}}}\chi_{\{\lambda\in [0, \infty): 
 \lambda^{\theta}\geq 1\}}.$
Then it is easy to see that $m_{\theta, \beta}\in \mathscr{M}(\theta, \beta).$
\end{example}

Now we state our main sparse domination principle for the multiplier class $\mathscr{M}(\theta, \beta).$
\begin{theorem}
\label{mainthm-1}
 Let $\theta\in \mathbb{R}\setminus \{0\}$ and $\beta\geq 0,$ and $m\in \mathscr{M}(\theta, \beta).$ Then there exist sparse families $\mathcal{S}$ and $\mathcal{S}'$ such that for all compactly supported bounded functions $f, g$ we have
 \begin{align*}
 |\langle m(\sqrt{\La})f, g \rangle |\lesssim_{\theta, \beta, r_{1}, r_{2}}\sum_{R\in \mathcal{S}} |R|\langle f \rangle_{r_1, R}\Ln g\Rn_{r_2', R}\\
 \text{and}\,\, |\langle m(\sqrt{\La})f, g \rangle |\lesssim_{\theta, \beta, r_{1}, r_{2}}\sum_{R\in \mathcal{S'}} |R|\langle f \rangle_{r_2', R}\Ln g\Rn_{r_1, R}, 
 \end{align*}
 where $r_1, r_2$ satisfy 
 \begin{align}
 \ \ \ &\left(\frac{1}{r_1}-\frac{1}{2}\right)<\frac{\beta}{2Q},\ \ \ 1\leq r_1\leq r_2\leq 2,\label{sparse1}\\
 \nonumber\text{or}\\
 \ \ \  &\label{sparse2}\left(\frac{1}{r_1}-\frac{1}{r_2}\right)<\frac{\beta}{2Q},\ \ \ 1\leq r_1\leq 2\leq r_2\leq r'_{1}.
 \end{align}
\end{theorem}
As a special case, we obtain the following corollary.
\begin{corollary}
\label{mainthm-2}
 Let $\theta\in \mathbb{R}\setminus \{0\}$ and $\beta\geq 0.$ Then there exist sparse families $\mathcal{S}$ and $\mathcal{S}'$ such that for all compactly supported bounded functions $f, g$ we have
 \begin{align*}
 |\langle m_{\theta, \beta}(\sqrt{\La})f, g \rangle |\lesssim_{\theta, \beta, r_{1}, r_{2}}\sum_{R\in \mathcal{S}} |R|\langle f \rangle_{r_1, R}\Ln g\Rn_{r_2', R}\\
 \text{and}\,\, |\langle m_{\theta, \beta}(\sqrt{\La})f, g \rangle |\lesssim_{\theta, \beta, r_{1}, r_{2}}\sum_{R\in \mathcal{S'}} |R|\langle f \rangle_{r_2', R}\Ln g\Rn_{r_1, R}, 
 \end{align*}
 where $r_1, r_2$ satisfy either condition \eqref{sparse1} or \eqref{sparse2}.
\end{corollary}
The motivation for proving such an estimate arises from recent works \cite{BFP, BC, Conde, HytA2, Hyt17, Lacey17, Laceysph, LM, Lerner-NYJM} where sparse domination is achieved for several classical operators in Harmonic analysis in various settings. The key importance in proving such an estimate lies in the fact that one can obtain a range of quantitative weighted estimates depending on the decay parameter $\beta,$ we state them here. The Muckenhoupt class of weights ($A_p$) and reverse H\"older's classes ($RH_{q}$) are defined in details in Section~\ref{Appl}.
\begin{theorem}
\label{quantitative}
Let $\theta\in \mathbb{R}\setminus \{0\}.$ We have the following results:
\begin{enumerate}[i)]
    \item 
Let $m\in \mathscr{M}(\theta, 2Q).$ Then $m(\sqrt{\La})$ maps $L^p(\omega)$ to $L^p(\omega)$ for all $1<p<\infty$ and $\omega\in A_{p}.$ 
\item
Let $m\in \mathscr{M}(\theta, \beta)$ with $Q\leq \beta< 2Q.$ Then $m(\sqrt{\La})$ maps $L^p(\omega)$ to $L^p(\omega)$ for $p_{\beta}<p<\infty$ and $\omega\in A_{p/p_{\beta}},$ where $p_{\beta}:=\frac{2Q}{\beta}.$
\item
Let $m\in \mathscr{M}(\theta, \beta)$ with $0<\beta<Q.$ Then $m(\sqrt{\La}): L^p(\omega)\to L^p(\omega)$ for all $2<p<s_{\beta},\ \  \omega\in A_{p/2}\cap RH_{(s_{\beta}/p)'},$ where $\frac{1}{s_{\beta}}:=\frac{1}{2}-\frac{\beta}{2Q}.$
\end{enumerate}
\end{theorem}
We believe these results are completely new in the setting of stratified Lie groups. The article is organized as follows. In the next section, we recall some necessary preliminaries and Section~\ref{main} contains the proof of Theorem~\ref{mainthm-1}. In Section~\ref{Appl}, we prove Theorem~\ref{quantitative} and other applications of Theorem~\ref{mainthm-1} to Riesz means and dispersive equations.
\section{Preliminaries}
Let $\{\delta_r\}_{r>0}$ be the group of dilations associated to $G$ and let $|\cdot|$ be a homogeneous quasi-norm, i.e.,  $|x| = 0$ if and only if $x=0$ where $0$ denotes the group identity, and $|\delta_r x | = r |x|$ for all $r>0$ and $x\in G$. Moreover, the right convolution kernel of the operator $m(\sqrt{\La})$ will be denoted by $K_{m},$ that is, $$m(\sqrt{\La})f(x)=\ \int_G f(x\cdot y^{-1}) \,
K_m (y) \, dy=\int f(y)K_{m}(y^{-1}x)\ dy.$$
In general, $K_{m}$ is just a distribution but whenever $m$ is compactly supported, $K_m$ can be identified with an $L^2$ function on $G.$ See \cite{FR} for more details regarding analysis on these groups. The following estimates are well known. 
\begin{theorem}[\cite{sikora}]
 The following kernel estimates are either known. For any function $h$ and $R>0$, we denote $h_{R}(t):=h(tR).$ 
 \begin{enumerate}[i)]
     \item 
The following Plancherel-type identity holds
 \begin{align}
 \label{pe-1}
  \|K_{h}\|^{2}_{L^2(G)}=\int_{0}^{\infty} |h(t)|^2 t^{Q-1}\, dt.    
 \end{align}    
 In particular, if the multiplier is supported on $[0, R]$ then $\|K_{h}\|^{2}_{L^2(G)}\leq R^{Q}\|h_{R}\|_{2}^{2}.$
 \item
 For any compactly supported multiplier $h,$
 \begin{equation}\label{pe-2}
\int_G |K_{h}(x)|^2 (1+|x|^s)^2 \, dx \ \lesssim \ \|h\|^2_{L^2_s(\mathbb{R}_{+})}
\end{equation}
holds for any $s > 0$. As a consequence, $\|K_{h}\|_{L^1}\leq \|h\|_{L^2_{s}(\mathbb{R}_{+})}$ for $s>\frac{Q}{2}.$
 \end{enumerate}
\end{theorem}

We also need the following notion of dyadic grids in spaces of homogeneous type. We refer to \cite{Chr90} and \cite{HytKai, Lor21} for details. Let $0<c_1\leq C_1<\infty$ and $\mu\in (0,1)$. By a general dyadic grid $\mathscr{D}=\bigcup_{k\in \mathbb{Z}}\mathscr{D}_k$ on $G$, we mean a countable collection of sets ${R}_k^{\alpha}$ for $k\in\mathbb{Z}$, each associated with a point $z_k^\alpha$, $\alpha$ coming from a countable index set, with the following properties:
\begin{itemize}
    \item $G=\bigcup\limits_{\alpha}{R}_k^{\alpha}$ for every $k\in\mathbb{Z}$.
    \item If $l\geq k$, then either $R_l^{\beta}\subset R_k^{\alpha}$ or $R_l^{\beta}\cap R_k^{\alpha}=\emptyset$.
    \item For the constants $c_1, C_1>0$ we have $B(R_k^\alpha):=B(z_k^\alpha,c_1\mu^k)\subset R_k^\alpha\subset B(z_k^\alpha,C_1\mu^k)=:C_{1} B(R_k^\alpha)$.
    \item If $l\geq k$ and $R_l^{\beta}\subset R_k^{\alpha}$, then $C_1 B(R_l^{\beta})\subset C_1 B(R_k^{\alpha})$.
\end{itemize}
Hyt\"onen and Kairema \cite[Theorem 4.1]{HytKai} proved the existence of a finite collection of dyadic grids $\mathscr{D}^{n},$ $n=1,2,\ldots,\mathfrak{N},$ such that for every ball $B(z,r)\subset G$ with $\mu^{k+2}\leq r<\mu^{k+1}$, there exists some $n\in \{1,2,\ldots,\mathfrak{N}\}$ and $R_k^\alpha\in \mathscr{D}^n$ such that $B(z,r)\subset R_k^\alpha$ and $\text{diam}\,(R_k^\alpha)\leq C\,r$, where $C$ depends on $\mu$. For the purposes of this article, the number $\mu\in (0, 1)$ is now considered fixed and $\nu$ will denote $\frac{1}{\mu}$. 

\begin{remark}
We remark that the sparse families $\mathcal{S}, \mathcal{S'}$ in Theorem~\ref{mainthm-1} consist of elements from the dyadic grids $\mathscr {D}^n,$ $n=1,2,\ldots,\mathfrak{N}.$
\end{remark}

\section{Proof of Theorems}
\label{main}
We shall prove Theorem~\ref{mainthm-1} for case $\theta>0.$ Let us fix $\theta>0, \beta\geq 0,$ and $m\in \mathscr{M}(\theta, \beta).$ Recall that $m^j(\lambda) := m(\nu^j \lambda) \phi(\lambda)$ for $j\geq 0,$ where $\phi$ is a smooth function on $(0, \infty),$ supported
on $\{\nu^{-1}\leq \lambda\leq \nu\},$ satisfying $\sum_{j} \phi(\nu^{-j}\lambda)=1.$ Then, $m^j$ satisfy the following
\begin{align}
\label{hypo-1}
& \|m^j\|_{L^{\infty}({\mathbb R}_{+})}\lesssim \nu^{-j \theta \beta/2}\ \ \ \text{for}\ \ j\geq 0,\\
\nonumber \text{and}\ \ &\|m^j\|_{L^2_s({\mathbb R}_{+})}\lesssim \nu^{ j \theta (2s - \beta) /2}\ \ \ \text{for}\ \ j\geq 0,
\label{hypo-1}
\end{align}
where the implicit constants are independent of $j.$ We also introduce the following notation $m_{j}(\lambda):=m(\lambda) \phi(\nu^{-j}\lambda).$ Recall that $\sum_{j} \phi(\nu^{-j}\lambda)=1,$ then $m=\sum_{j\geq 0} m_{j}(\lambda).$ Moreover, $m_{j}(\lambda)=m^{j}(\nu^{-j}\lambda).$ Then we have the following decomposition
\begin{align*}
m(\sqrt{\La})=\sum_{j\geq 0} T_{j}, \ \ \ \ \text{where}\ \ \ \ T_{j}f=f* K_{m_{j}}.    
\end{align*}
 It is easy to see by homogeneity that $K_{m_{j}}=K_{m}* (\nu^{jQ}K_{\phi}(\delta_{\nu^j}\cdot))(x).$ Motivated by \cite{BC} we make a further decomposition in the space variable, namely
 \begin{align}
  T^{l}_{j}f(x)=\int f(z)K_{m_{j}}(z^{-1}x) \phi(\nu^{-l+j(1-\theta)}|z^{-1}x|) \ dz,\ \ \ l\in \mathbb{Z}.    
 \end{align}
Then $T_{j}f=\sum_{l\in \mathbb{Z}}T^{l}_{j}f$ and consequently $m(\sqrt{\La})=\sum_{j\geq 0}\sum_{l\in \mathbb{Z}} T^{l}_{j}.$ Now we shall focus on proving certain crucial estimates and for some $\epsilon>0$ we group the terms according to their spatial scale, i.e., $$T_{j}f=\sum_{l\leq j\epsilon} T^{l}_{j}f+ \sum_{l>j \epsilon} T^{l}_{j}f.$$

Let us start by proving $L^2-L^2$ estimates for the pieces $T^{l}_{j}.$ Let $l>j\epsilon$ and denote $g:=K_{m_{j}}(\cdot) \phi(\nu^{-l+j(1-\theta)}|\cdot|).$
By Young's inequality we have
\begin{align*}
 \|T_{j}^{l}f\|_{2}\leq \|f\|_{2}\|g\|_{1}
 &\leq \|f\|_{2}\left(\int_{|x|\simeq \nu^{l-j(1-\theta)}}|K_{m_{j}}(x)|\, dx \right)\\
 &= \|f\|_{2}\left(\int_{|x|\simeq \nu^{l-j(1-\theta)}} \nu^{jQ}|K_{m^{j}}(\delta_{\nu^{j}}x)|\, dx \right)\\
 &\leq \|f\|_{2}\left(\int_{|x|\simeq \nu^{l+j\theta}} |K_{m^{j}}(x)|\, dx \right)\\
 &=\|f\|_{2}\left(\int_{|x|\simeq \nu^{l+j\theta}}(1+|x|^s) (1+|x|^s)^{-1}|K_{m^{j}}(x)|\, dx \right)\\
 &\leq \|f\|_{2}\left(\int_{|x|\simeq \nu^{l+j\theta}}(1+|x|^s)^{2} |K_{m^{j}}(x)|^{2}\, dx \right)^{1/2} \nu^{(l+j\theta)(\frac{Q}{2}-s)}\\
 &\lesssim  \nu^{(l+j\theta)(\frac{Q}{2}-s)} \|m^j\|_{L^2_{s}} \|f\|_{2}\lesssim  \nu^{(l+j\theta)(\frac{Q}{2}-s)}  \nu^{j\theta(s-\frac{\beta}{2})} \|f\|_{2}\\
 &\lesssim  \nu^{-\frac{j\theta \beta}{2}} \nu^{l(\frac{Q}{2}-\frac{s}{2})} \nu^{\frac{j\theta Q}{2}} \nu^{-l\frac{s}{2}}\|f\|_{2}\lesssim  \nu^{l(\frac{Q}{2}-\frac{s}{4})} \nu^{-\frac{ls}{4}} \nu^{\frac{j\theta Q}{2}} \nu^{-l\frac{s}{2}}\|f\|_{2}.
\end{align*}
Observe that the term $\nu^{l(\frac{Q}{2}-\frac{s}{4})}\ll 1$ if $s$ is chosen large enough. Moreover, as $l>j\epsilon,$ $\nu^{-\frac{ls}{4}} \nu^{\frac{j\theta Q}{2}}\leq \nu^{\frac{j\theta Q}{2}} \nu^{-\frac{j\epsilon s}{4}}\ll 1$ provided $s\gg \lfloor{\frac{Q\theta}{\epsilon}\rfloor}.$ Finally, choose $s$ large such that $\nu^{-l\frac{s}{4}}\leq \nu^{-Q(Q+\frac{\theta\beta}{2})l}$ as well as $\nu^{-l\frac{s}{4}}\leq \nu^{-j\epsilon\frac{s}{4}}\leq \nu^{-Q(Q+\frac{\theta\beta}{2})j}.$ Therefore, we obtain
\begin{align}
\label{l2-in}
\|T^{l}_{j}f\|_{2}\leq c_{\epsilon} \nu^{- Q(Q+\frac{\theta\beta}{2})(j+l)}\|f\|_{2} \ \ \ \text{for} \ \ l> j\epsilon. 
\end{align}
Certainly, $\|T_{j}f\|_{2}\leq \nu^{-\frac{j\theta \beta}{2}}\|f\|_{2}$ since $\|m_{j}\|_{L^{\infty}}\leq \nu^{-\frac{j\theta \beta}{2}}$. Combining this with \eqref{l2-in} we prove the following lemma:
\begin{lemma}
\label{l2l2}
We obtain the following estimates:
 \begin{enumerate}[i)]
     \item 
 \begin{equation*}
 \label{l2l2-big}
 \|T^{l}_{j}f\|_{2}\lesssim_{\epsilon} \nu^{-Q(Q+\frac{\theta\beta}{2})(j+l)}\|f\|_{2} \ \ \ \text{for} \ \ l> j\epsilon.    
 \end{equation*} 
 \item
 \begin{equation*}
  \label{l2l2-small}
 \|\sum_{l\leq j\epsilon}T^{l}_{j}f\|_{2}\lesssim_{\epsilon} \nu^{-\frac{j\theta \beta}{2}} \|f\|_{2}.
 \end{equation*}
 \end{enumerate}   
\end{lemma}
\begin{remark}
\label{Rmk}
 At this point we would like to remark that one can in fact improve the bounds for $T^{l}_{j}, l> j\epsilon,$ to $\|T^{l}_{j}f\|_{2}\lesssim_{\epsilon} \nu^{-c_{\theta} Q(Q+\frac{\theta\beta}{2})(j+l)}\|f\|_{2}$
 for some large constant $c_{\theta}$. The similar remark is also applicable for Lemma~\ref{l1l1}, Lemma~\ref{linf}, Lemma~\ref{lsls}, Lemma~\ref{lrls}, and Lemma~\ref{improv} but we refrain ourselves from mentioning it repetitively. 
\end{remark}
Now we shall prove $L^1-L^1$ estimates for the pieces $T^{l}_{j}.$ Recall that $g=K_{m_{j}}(\cdot) \phi(\nu^{-l+j(1-\theta)}|\cdot|).$ The previous argument shows that for any $l> j\epsilon$
\begin{equation}
\label{l1l1-in}
\|T^{l}_{j}f\|_{1}\leq \|f\|_{1} \|g\|_{1}\lesssim_{\epsilon} \nu^{-Q(Q+\frac{\theta\beta}{2})(j+l)}\|f\|_{1}.    
\end{equation}

Another observation, together with \eqref{pe-2}, yields the following for any $s>\frac{Q}{2}$ 
\begin{equation}
\label{reqL1}
 \|T_{j}f\|_{1}\leq \|f\|_{1} \|K_{m_j}\|_{1}\leq \|f\|_{1} \|K_{m^j}\|_{1}\leq \nu^{j\theta(s-\frac{\beta}{2})}\|f\|_{1}.
\end{equation}
Consequently, $$\|T_{j}f\|_{1}\leq \nu^{j(-\frac{\theta\beta}{2}+\frac{\theta Q}{2}+\varepsilon)}\|f\|_{1}\ \ \ \text{for any}\ \ \varepsilon>0.$$
Moreover, summing \eqref{l1l1-in} in $l,$ we obtain $$\|\sum_{l\leq j\epsilon}T^{l}_{j}f\|_{1}=\|T_{j}f-\sum_{l> j\epsilon}T^{l}_{j}f\|_{1}\lesssim \nu^{j(-\frac{\theta\beta}{2}+\frac{\theta Q}{2}+\varepsilon)}\|f\|_{1}.$$
\begin{lemma}
\label{l1l1}
Combining \eqref{l1l1-in} and the above discussion we have the following estimates:
\begin{enumerate}[i)]
\item
\label{l1l1-big}
For any $\varepsilon>0$ 
\begin{equation*}
\|\sum_{l\leq j\epsilon}T^{l}_{j}f\|_{1}\lesssim_{\epsilon} \nu^{j(-\frac{\theta\beta}{2}+\frac{\theta Q}{2}+\varepsilon)}\|f\|_{1}.   
\end{equation*}
    \item 
    \label{l1l1-small}
For any $l> j\epsilon,$ we have
\begin{equation*}
\|T^{l}_{j}f\|_{1}\lesssim_{\epsilon} \nu^{-Q(Q+\frac{\theta\beta}{2})(j+l)}\|f\|_{1}.    
\end{equation*}
\end{enumerate}
\end{lemma}
Finally, we need $L^1-L^{\infty}$ estimates for the operators $T^{l}_{j}.$ In order to do that we first need to prove pointwise estimates for the kernel $K_{h}$ of the operator $h(\sqrt{\mathcal{L}})$ where $h$ is supported on $[R/4, R]$ for some $R>0.$ Let $p_t$ denote the convolution kernel associated with the heat operator $e^{-t\mathcal{L}}.$ Recall the following Gaussian estimate (see e.g. \cite{heat})
\begin{equation}
\label{gauss}
 |p_{t}(x)|\leq \frac{C}{t^{Q/2}} e^{-\frac{|x|^2}{c\,t}}.   
\end{equation}
We just sketch the proof, see \cite{BuiIMRN} for details. Denote $H(\lambda)=e^{\frac{\lambda^2}{R^2}}h(\lambda),$ then $h(\lambda)=e^{-\frac{\lambda^2}{R^2}}H(\lambda).$ Also, $\|H_{R}\|_{2}\simeq \|h_{R}\|_2$ as $h$ is supported on $[R/4, R].$ Therefore, 
$K_{h}(y^{-1}x)=\int_{G} p_{\frac{1}{R^2}}(z^{-1} x) K_{H}(y^{-1}z)\, dz.$ H\"older's inequality and \eqref{pe-1} implies
\begin{align}
\label{interpol-1}
 |K_{h}(y^{-1}x)|\leq \|p_{1/R^2}(z^{-1}x)\|_{L^2(dz, G)}\|K_{H}(y^{-1}z)\|_{L^2(dz, G)}\lesssim R^{Q} \|H_{R}\|_{2}\lesssim R^{Q} \|h_{R}\|_{2}.
\end{align}
Observe that a factor of $R^{Q/2}$ appears from \eqref{pe-1} and another from $R^{Q/2}$  from $\|p_{1/R^2}(z^{-1}x)\|_{L^2(dz, G)}.$ Using Fourier inversion, we can write \begin{align}
\label{kernelRep}
K_{h}(y^{-1}x)=\frac{1}{2\pi}\int_{\mathbb{R}} \widehat{G}(t)p_{(1-it)/R^2}(y^{-1}x)\, dt,
\end{align}
where $G(\lambda):=h(R\sqrt{\lambda})e^{\lambda}.$ Note that $supp (G)$ is contained $[0, 1]$ due to the support condition on $h$, so $e^\lambda$ and its derivatives are always bounded. At this point we use the following estimate from \cite{Ouha}
$$|p_{(1-it)/R^2}(y^{-1}x)|\leq C R^{Q}e^{-\frac{R^2 |y^{-1}x|^2}{(1+t^2)}} \leq C R^{Q} (1+R|y^{-1}x|)^{-s}(1+|t|)^s.$$
Therefore, from the above bound with \eqref{kernelRep}, we have for any $s>0$
\begin{align}
 \nonumber |K_{h}(y^{-1}x)|  \leq C R^{Q}(1+R|y^{-1}x|)^{-s}\int |\widehat{G}(t)|(1+|t|)^{s}\, dt &\lesssim R^{Q}(1+R|y^{-1}x|)^{-s} \|G\|_{L^{2}_{s+\varkappa+\frac{1}{2}}}\\
 &\lesssim R^{Q}(1+R|y^{-1}x|)^{-s} \|h_{R}\|_{L^{2}_{s+\varkappa+\frac{1}{2}}},\label{interpol-2}
\end{align}
for any small $\kappa>0.$ Using complex interpolation of \eqref{interpol-1} and \eqref{interpol-2}, as in \cite{Duong02}, we remove the extra exponent $1/2$ in the Sobolev exponent to obtain 
\begin{align}
\label{pointwise}
|K_{h}(y^{-1}x)|\lesssim R^{Q}(1+R|y^{-1}x|)^{-s} \|h_{R}\|_{L^{2}_{s+\varkappa}}   
\end{align}
for any $s>0$ and any arbitrarily small $\varkappa>0.$ Now we prove the following lemma regarding $L^1-L^{\infty}$ estimates for $T^{l}_{j}.$
\begin{lemma}
\label{linf}
\begin{enumerate}[i)]
    \item 
 \label{linf-small}
 For $l> j\epsilon$ 
$$\|T^{l}_{j}f\|_{L^{\infty}}\lesssim_{\epsilon} \nu^{-Q(Q+\frac{\theta\beta}{2})(j+l)} \|f\|_{1}.$$
 \item We also have
 \label{linf-big}
$$\|\sum_{l\leq j\epsilon}T^{l}_{j}f\|_{\infty}\leq \nu^{jQ} \nu^{-\frac{j\theta Q}{2}} \|f\|_{1}.$$
\end{enumerate}
\end{lemma}
\begin{proof}
 Let $l>j\epsilon.$ Then we have
 \begin{align*}
 |T^{l}_{j}f(x)|&\leq \left(\sup_{\{y: |y^{-1}x|\simeq \nu^{l-j(1-\theta)}\}}|K_{m_j}(y^{-1}x)|\right)\int |f(y)|\, dy\\
 &\leq \left(\sup_{\{y: |y^{-1}x|\simeq \nu^{l-j(1-\theta)}\}}|\nu^{jQ}K_{m^j}(\delta_{\nu^j}(y^{-1}x))|\right)\|f\|_{1}\\
 &\leq \nu^{jQ} \left(\sup_{\{y: |y^{-1}x|\simeq \nu^{l+j\theta}\}}|K_{m^j}((y^{-1}x))|\right)\|f\|_{1}\\
 &\leq \nu^{jQ} \left(\sup_{\{y: |y^{-1}x|\simeq \nu^{l+j\theta}\}}(1+|y^{-1}x|)^{-s} \|m^{j}\|_{L^{2}_{s+\varkappa}}\right)\|f\|_{1}\ \ (\text{using} \eqref{pointwise})\\
 &\lesssim \nu^{jQ} \nu^{-s(l+j\theta)} \nu^{j\theta(s+\varkappa-\frac{\beta}{2})}\|f\|_{1},
 \end{align*}    
 for any $s>Q$ and a fixed small $\varkappa>0$ from \eqref{pointwise}. Since, $l> j\epsilon,$ we may choose $s$ sufficiently large, as done in the proof of Lemma~\ref{l2l2},  depending on $\epsilon, \varkappa$ such that
 \begin{align*}
\|T^{l}_{j}f\|_{L^{\infty}}\lesssim_{\epsilon} \nu^{-Q(Q+\frac{\theta\beta}{2})(j+l)} \|f\|_{1}.  
 \end{align*}
For the second part, observe that 
\begin{align*}
\big |\sum_{l\leq j\epsilon}T^{l}_{j}f(x)\big|=&\big| \int f(z) K_{m_{j}}(z^{-1}x) \sum_{l\leq j\epsilon} \phi(\nu^{-l+j(1-\theta)}|z^{-1}x|)\big|\\
&\leq \int |f(z)| |K_{m_{j}}(z^{-1}x)|\, dz\\
&\overset{\text{using}\,\,\eqref{interpol-1}}\lesssim\nu^{jQ}\|m^j\|_{L^{\infty}}\|f\|_{1}\lesssim \nu^{jQ} \nu^{-\frac{j\theta Q}{2}} \|f\|_{1},    
\end{align*}
completing the proof.
 \end{proof} 

\begin{lemma}[$L^{r_1}-L^{r_1}$ estimates]
\label{lsls}
 Let $1\leq r_1\leq 2.$ Then interpolating Lemma~\ref{l1l1} and Lemma~\ref{l2l2} we obtain the following:
 \begin{enumerate}[i)]
     \item 
For any $\varepsilon>0$     
\begin{align}
 \label{ls-1}
 \|\sum_{l\leq j\epsilon} T^{l}_{j}f\|_{r_1}\lesssim_{\epsilon}\nu^{-\frac{j\theta\beta}{2}} \nu^{j(\frac{\theta Q}{2}+\varepsilon)(\frac{2}{r_1}-1)}\|f\|_{r_1}.
 \end{align}     
 \item
 For any $l>j\epsilon$
 \begin{align}
 \label{ls-2}
 \|T^{l}_{j}f\|_{r_1}\lesssim_{\epsilon} \nu^{-Q(Q+\frac{\theta\beta}{2})(j+l)}\|f\|_{r_1}.
 \end{align}
\end{enumerate}
\end{lemma}
The next lemma concerns the key estimate which is required for our sparse domination estimates. 
\begin{lemma}
\label{lrls}
 Let $1\leq r_1\leq r_2\leq 2.$ Then we have the following estimates:
 \begin{enumerate}[i)]
     \item 
For $l> j\epsilon$     
\begin{align}
    \label{lrls1}
\|T^{l}_{j}f\|_{r_2}\lesssim_{\epsilon}\nu^{-\frac{j\theta\beta}{2}} \nu^{jQ\left(\frac{1}{r_1}-\frac{1}{r_2}\right)} \|f\|_{r_1}    
\end{align}     
\item
For any $\varepsilon>0$
\begin{align}
\label{lrls2}
 \|\sum_{l\leq j\epsilon} T^{l}_{j}f\|_{r_2}\lesssim_{\epsilon}\nu^{-\frac{j\theta\beta}{2}} \nu^{j(\frac{\theta Q}{2}+\varepsilon)(\frac{2}{r_2}-1)}  \nu^{jQ\left(\frac{1}{r_1}-\frac{1}{r_2}\right)}\|f\|_{r_1}.     
\end{align}
 \end{enumerate}
\end{lemma}
\begin{proof}
 Let us introduce a smooth cutoff function $\psi$ such that $\psi=1$ on the support of $\phi.$ Then $m_{j}=m\,\phi(\nu^{-j}\lambda)=m\,\phi(\nu^{-j}\lambda) \psi(\nu^{-j}\lambda)=m_{j}\,\psi(\nu^{-j}\lambda).$ Let $K_{\psi}$ be the kernel of $\psi(\sqrt{\La}).$ Therefore, for $l> j\epsilon,$ using Lemma~\ref{lsls} and Young's inequality, we obtain 
 \begin{align}
\label{young-1}
\|T^{l}_{j}f\|_{r_2}\lesssim_{\epsilon}\nu^{-Q(Q+\frac{\theta\beta}{2})(j+l)}\|f*K_{\psi({\nu^{-j}\cdot})}\|_{r_2}\lesssim_{\epsilon}\nu^{-Q(Q+\frac{\theta\beta}{2})(j+l)}\|f\|_{r_1} \|K_{\psi({\nu^{-j}\cdot})}\|_{t},    
 \end{align} 
 where $\frac{1}{t}=\frac{1}{r_2}+1-\frac{1}{r_1}.$ Next recall from \eqref{pointwise} and homogeneity that $$|K_{\psi(\nu^{-j}\cdot)}(x)|=\nu^{jQ}|K_{\psi}(\delta_{\nu^j}x)|\leq \nu^{jQ}\frac{C}{(1+|\delta_{\nu^j}x|)^N}\leq \nu^{jQ}\frac{C}{(1+\nu^{j}|x|)^N}$$ for any $N>0.$ 
From \eqref{young-1} and the above pointwise estimate we obtain the following
\begin{align*}
\|T^{l}_{j}f\|_{r_2}&\lesssim_{\epsilon}\nu^{-Q(Q+\frac{\theta\beta}{2})(j+l)}\|f*K_{\psi({\nu^{-j}\cdot})}\|_{r_2}\\
&\lesssim_{\epsilon}\nu^{-Q(Q+\frac{\theta\beta}{2})(j+l)}\|f\|_{r_1} \|K_{\psi({\nu^{-j}\cdot})}\|_{t}\lesssim_{\epsilon}\nu^{-Q(Q+\frac{\theta\beta}{2})(j+l)} \nu^{jQ\left(\frac{1}{r_1}-\frac{1}{r_2}\right)} \|f\|_{r_1},    
\end{align*}
where we have used that $\|K_{\psi({\nu^{-j}\cdot})}\|_{t}\lesssim \nu^{jQ(1-\frac{1}{t})}$ by choosing $N$ sufficiently large. Similarly, modifying the above arguments together with \eqref{ls-1}, we have
\begin{align*}
 \|\sum_{l\geq j\epsilon} T^{l}_{j}f\|_{r_2}\lesssim_{\epsilon}\nu^{-\frac{j\theta\beta}{2}} \nu^{j(\frac{\theta Q}{2}+\varepsilon)(\frac{2}{r_2}-1)} \|f\|_{r_1} \nu^{jQ\left(\frac{1}{r_1}-\frac{1}{r_2}\right)}
\end{align*}
for any $\varepsilon>0.$
\end{proof}
We also need the following $L^p$ improving estimate.
\begin{lemma}
\label{improv}
Let $1\leq r_1\leq 2\leq r_2\leq r_{1}'.$ Then the following estimates hold true:
\begin{enumerate}[i)]
    \item For $l> j\epsilon$
\begin{align}
\label{sum1}
\|T^{l}_{j}f\|_{r_2}\lesssim_{\epsilon} \nu^{-Q(Q+\frac{\theta\beta}{2})(j+l)} \nu^{jQ(\frac{1}{r_1}-\frac{1}{r_2})}\|f\|_{r_1}.
\end{align}
\item
We also have 
\begin{align}
\label{sum2}
\|\sum_{l\leq j\epsilon} T^{l}_{j}f\|_{r_2}\lesssim_{\epsilon} \nu^{-\frac{j\theta\beta}{2}}\nu^{jQ(\frac{1}{r_1}-\frac{1}{r_2})}\|f\|_{r_1}.    
\end{align}  
\end{enumerate}    
\end{lemma}
\begin{proof}
Interpolating Lemma~\ref{l2l2} and Lemma~\ref{linf} we obtain that for any $1\leq r_1\leq 2$
\begin{align}
\label{lrlr'-1}
 &\|T^{l}_{j}f\|_{r_1'}\lesssim_{\epsilon} \nu^{-Q(Q+\frac{\theta \beta}{2})(j+l)}\|f\|_{r_1}\ \ \text{for}\ \ l\geq j\epsilon,\\
 \text{and}\ \ \ &\|\sum_{l\leq j\epsilon}T^{l}_{j}f\|_{r_1'}\lesssim_{\epsilon} \nu^{-\frac{j\theta Q}{2}}\nu^{jQ(1-\frac{2}{r'_1})}\|f\|_{r_1}. 
 \label{lrlr'-2}
\end{align}
Recall from the previous lemma that $m_{j}=m_{j}\psi(\nu^{-j}\lambda)$ where $\psi$ is introduced in Lemma~\ref{lrls}. Also recall \begin{align}
\label{pt-req}
|K_{\psi(\nu^{-j}\cdot)}(x)|\leq \nu^{jQ}\frac{C}{(1+\nu^j|x|)^N}\ \  \text{for any}\ \  N>0.
\end{align}
For $l> j\epsilon,$ employing Young's inequality we obtain
\begin{align*}
 \|T^l_{j}f\|_{2}\lesssim_{\epsilon} \nu^{-Q(Q+\frac{\theta\beta}{2})(j+l)}\|f*K_{\psi(\nu^{-j}\cdot)}\|_{2}\leq \nu^{-Q(Q+\frac{\theta\beta}{2})(j+l)} \|K_{\psi(\nu^{-j}\cdot)}\|_{t}\|f\|_{r_1},  
\end{align*}
where $\frac{1}{t}=\frac{1}{2}+\frac{1}{r_1'}.$ Therefore, combining the above with \eqref{pt-req} yields the following for $l> j\epsilon$ 
\begin{align}
\|T^{l}_{j}f\|_{2}\lesssim_{\epsilon} \nu^{-Q(Q+\frac{\theta\beta}{2})(j+l)} \nu^{jQ(\frac{1}{r_1}-\frac{1}{2})}\|f\|_{r_1}. 
 \label{lrlr'-3}
\end{align}
Arguing similarly we obtain
\begin{align}
\label{lrlr'-4}
\|\sum_{l\leq j\epsilon} T^{l}_{j}f\|_{2}\lesssim_{\epsilon} \nu^{-\frac{j\theta\beta}{2}}\nu^{jQ(\frac{1}{r_1}-\frac{1}{2})}\|f\|_{r_1}.  
\end{align}
Interpolating \eqref{lrlr'-1} and \eqref{lrlr'-3} we obtain that
\begin{align*}
&\|T^{l}_{j}f\|_{r_2}\lesssim_{\epsilon} \nu^{-Q(Q+\frac{\theta\beta}{2})(j+l)} \nu^{jQ(\frac{1}{r_1}-\frac{1}{r_2})}\|f\|_{r_1}\ \ \ \ \text{holds for all}\ \ \ l> j\epsilon. 
\end{align*}
Similarly, interpolating \eqref{lrlr'-2} and \eqref{lrlr'-4}, we get
\begin{align*}
\|\sum_{l\leq j\epsilon} T^{l}_{j}f\|_{r_2}\lesssim_{\epsilon} \nu^{-\frac{j\theta\beta}{2}}\nu^{jQ(\frac{1}{r_1}-\frac{1}{r_2})}\|f\|_{r_1}.
\end{align*}
This completes the proof.
\end{proof}
Now we are in a position to prove our main Theorem~\ref{mainthm-2} for the case $\theta>0.$ After having the key unweighted estimates, the proof of sparse domination is now quite standard and we provide a brief sketch, for more details we refer to \cite{BC}. 
\begin{proof}[Proof of Theorem~\ref{mainthm-1}]
 Recall the dyadic families $\mathscr{D}^n$ for $n=1, \cdots, \mathfrak{N}$ and $\mathscr{D}^n=\cup_{k\in \mathbb{Z}} \mathscr{D}^{n}_{k}.$ Let us define the operators 
 \begin{align*}
 & T^{l}_{j, n}f:=\sum_{\substack{R\in \mathscr{D}^n:\\ R\in \mathscr{D}^n_{\lfloor j(1-\theta)-j\epsilon\rfloor}}} T^{l}_{j}(f\chi_{c_{G}B(R)}) \ \ \text{for}\ \ l\leq j\epsilon,\\
 &  T^{l}_{j, n}f:=\sum_{\substack{R\in \mathscr{D}^n:\\ R\in \mathscr{D}^n_{\lfloor j(1-\theta)-l\rfloor}}} T^{l}_{j}(f\chi_{c_{G}B(R)}) \ \ \text{for}\ \ l> j\epsilon,
 \end{align*}
 where the universal constant $c_{G}$ is chosen sufficiently small(by rescaling the metric) to ensure that the support of $T^{l}_{j}(f\chi_{c_{G}B(R)})$ is contained in $R.$ Therefore, it is enough to obtain sparse domination for $$\mathcal{T}_{n}:=\sum\limits_{\substack{j\geq 0\\l\in \mathbb{Z}}} T^{l}_{j, n},\,\, \text{for}\ \  n=1, \cdots, \mathfrak{N}.$$ Hence, we only prove sparse domination for one of the $\mathcal{T}_n$ and suppress the index $n$ for simplicity. By localisation and Lemma~\ref{lrls}, we obtain for $1\leq r_1\leq r_2\leq 2$
 \begin{align}
 \nonumber\big|\Ln \sum_{l\leq j\epsilon} T^{l}_{j}f, g\Rn\big|&\lesssim \sum_{j\geq 0} \sum_{\substack{R\in \mathscr{D}:\\ R\in \mathscr{D}_{\lfloor j(1-\theta)-j\epsilon\rfloor}}} \big\|\big(\sum_{l\leq j\epsilon} T^{l}_{j}(f\chi_{c_{G}B(R)})\big) \chi_{R} \big\|_{r_{2}}\|g \chi_{R}\|_{r'_2} \\
 \nonumber& \overset{\eqref{lrls2}}\lesssim_{\epsilon} \sum_{j\geq 0} \sum_{\substack{R\in \mathscr{D}:\\ R\in \mathscr{D}_{\lfloor j(1-\theta)-j\epsilon\rfloor}}} \nu^{-\frac{j\theta\beta}{2}} \nu^{j(\frac{\theta Q}{2}+\varepsilon)(\frac{2}{r_2}-1)}  \nu^{jQ\left(\frac{1}{r_1}-\frac{1}{r_2}\right)}\|f \chi_{R}\|_{r_1} \|g \chi_{R}\|_{r'_2}\\
 %\nonumber& \lesssim_{\epsilon} \sum_{j\geq 0} \sum_{\substack{R\in \mathscr{D}:\\ R\in \mathscr{D}_{\lfloor j(1-\theta)-j\epsilon+C\rfloor}}} \nu^{-\frac{j\theta\beta}{2}} \nu^{j(\frac{\theta Q}{2}+\varepsilon)(\frac{2}{r_2}-1)}  \nu^{jQ\left(\frac{1}{r_1}-\frac{1}{r_2}\right)}\|f \chi_{R}\|_{r_1} \|g \chi_{R}\|_{r'_2}\\
 \nonumber& \lesssim_{\epsilon} \sum_{j\geq 0} \sum_{\substack{R\in \mathscr{D}:\\ R\in \mathscr{D}_{\lfloor j(1-\theta)-j\epsilon\rfloor}}} |R|^{\frac{1}{r_1}+\frac{1}{r'_2}-1}\nu^{-\frac{j\theta\beta}{2}} \nu^{j(\frac{\theta Q}{2}+\varepsilon)(\frac{2}{r_2}-1)}  \nu^{jQ\left(\frac{1}{r_1}-\frac{1}{r_2}\right)}|R|\Ln f \Rn_{r_1, R} \Ln g\Rn_{r'_2, R}\\
 \nonumber& \lesssim_{\epsilon} \sum_{j\geq 0} \sum_{\substack{R\in \mathscr{D}:\\ R\in \mathscr{D}_{\lfloor j(1-\theta)-j\epsilon\rfloor}}} \nu^{Q(j\epsilon-j(1-\theta))\big(\frac{1}{r_1}+\frac{1}{r'_2}-1\big)}\nu^{-\frac{j\theta\beta}{2}} \nu^{j(\frac{\theta Q}{2}+\varepsilon)(\frac{2}{r_2}-1)}  \nu^{jQ\left(\frac{1}{r_1}-\frac{1}{r_2}\right)} |R|\Ln f \Rn_{r_1, R} \Ln g\Rn_{r'_2, R}\\
 & \lesssim_{\epsilon} \sum_{j\geq 0} \sum_{\substack{R\in \mathscr{D}:\\ R\in \mathscr{D}_{\lfloor j(1-\theta)-j\epsilon\rfloor}}}\nu^{j\big(\theta Q(\frac{1}{r_1}-\frac{1}{r_2})-\frac{\theta\beta}{2}+\frac{\theta Q}{2}\big(\frac{2}{r_2}-1\big)+\varepsilon\big(\frac{2}{r_2}-1\big)+\epsilon Q(\frac{1}{r_1}-\frac{1}{r_2})\big)} |R| \Ln f \Rn_{r_1, R} \Ln g\Rn_{r'_2, R},
 \label{sparse-1}
 \end{align}
 since $\epsilon$ and $\varepsilon$ are sufficiently small the above gives a geometrically decaying sparse collection if 
 $$\textstyle{\theta Q\big(\frac{1}{r_1}-\frac{1}{r_2}\big)-\frac{\theta\beta}{2}+\frac{\theta Q}{2}\big(\frac{2}{r_2}-1\big)<0\iff \frac{1}{r_1}-\frac{1}{2}<\frac{\beta}{2Q}.}$$

 A similar argument for $T^{l}_{j}, \ l> j\epsilon,$ with Remark~\ref{Rmk}, yields the following
 \begin{align}
 &\nonumber\big|\Ln \sum_{j\geq 0}\sum_{l>j\epsilon} T^{l}_{j}(f\chi_{c_{G}B(R)}), g\Rn\big|\\
 &\lesssim \sum_{j\geq 0} \sum_{l>j\epsilon}\sum_{\substack{R\in \mathscr{D}:\\ R\in \mathscr{D}_{\lfloor j(1-\theta)-l\rfloor}}} \| T^{l}_{j}(f\chi_{c_{G}B(R)}) \chi_{R} \|_{r_{2}}\|g \chi_{R}\|_{r'_2}\\
 \nonumber & \overset{\eqref{lrls1}}\lesssim \sum_{j\geq 0} \sum_{l>j\epsilon}\sum_{\substack{R\in \mathscr{D}:\\ R\in \mathscr{D}_{\lfloor j(1-\theta)-l\rfloor}}} \nu^{Q(l-j(1-\theta))\big(\frac{1}{r_1}-\frac{1}{r_2}\big)} \nu^{-c_{\theta}Q(Q+\frac{\theta\beta}{2})(j+l)} \nu^{jQ\left(\frac{1}{r_1}-\frac{1}{r_2}\right)} |R| \Ln f \Rn_{r_1, R} \Ln g\Rn_{r'_2, R}\\
 &\lesssim \sum_{j\geq 0} \nu^{jQ\big(\theta\big(\frac{1}{r_1}-\frac{1}{r_2}\big)-c_{\theta}(Q+\frac{\theta\beta}{2})\big)}\sum_{l>j\epsilon} \nu^{Ql\big(\big(\frac{1}{r_1}-\frac{1}{r_2}\big)-(Q+\frac{\theta\beta}{2})\big)} \sum_{\substack{R\in \mathscr{D}:\\ R\in \mathscr{D}_{\lfloor j(1-\theta)-l\rfloor}}} |R| \Ln f \Rn_{r_1, R} \Ln g\Rn_{r'_2, R}.
 \label{sparse-2}
 \end{align}
 As mentioned in Remark~\ref{Rmk} the constant $c_{\theta}$ can be chosen sufficiently large, hence, we can always ensure that $\big(\theta \big(\frac{1}{r_1}-\frac{1}{r_2}\big)-c_{\theta}(Q+\frac{\theta\beta}{2})\big)<0$ and $\big(\big(\frac{1}{r_1}-\frac{1}{r_2}\big)-(Q+\frac{\theta\beta}{2})\big)<0.$ Therefore, we again obtain geometrically decaying  $(r_1, r'_2)$ sparse domination. Therefore, combining \eqref{sparse-1} and \eqref{sparse-2}, we obtain $(r_1, r_2')$ sparse domination for $1\leq r_1\leq r_2\leq 2$ provided $\frac{1}{r_1}-\frac{1}{2}<\frac{\beta}{2Q}.$

 Arguing similarly in the case $1\leq r_1\leq 2\leq r_2\leq r_{1}'$ with Lemma~\ref{improv} yields
 \begin{align}
 \nonumber\big|\Ln \sum_{l\leq j\epsilon} T^{l}_{j}(f\chi_{c_{G}B(R)}), g\Rn\big|&\lesssim \sum_{j\geq 0} \sum_{\substack{R\in \mathscr{D}:\\ R\in \mathscr{D}_{\lfloor j(1-\theta)-j\epsilon\rfloor}}} \|\left(\sum_{{l\leq j\epsilon}} T^{l}_{j}(f\chi_{c_{G}B(R)})\right) \chi_{R} \|_{r_{2}}\|g \chi_{R}\|_{r'_2} \\
 \nonumber& \overset{\eqref{sum2}}\lesssim_{\epsilon} \sum_{j\geq 0} \sum_{\substack{R\in \mathscr{D}:\\ R\in \mathscr{D}_{\lfloor j(1-\theta)-j\epsilon\rfloor}}} \nu^{Q(j\epsilon-j(1-\theta))\big(\frac{1}{r_1}-\frac{1}{r_2}\big)}\nu^{-\frac{j\theta\beta}{2}}\nu^{jQ(\frac{1}{r_1}-\frac{1}{r_2})} |R|\Ln f \Rn_{r_1, R} \Ln g\Rn_{r'_2, R}\\
 & \lesssim_{\epsilon} \sum_{j\geq 0} \sum_{\substack{R\in \mathscr{D}:\\ R\in \mathscr{D}_{\lfloor j(1-\theta)-j\epsilon\rfloor}}}\nu^{j\big(-Q(1-\theta)\big(\frac{1}{r_1}-\frac{1}{r_2}\big)-\frac{\theta\beta}{2}+Q\big(\frac{1}{r_1}-\frac{1}{r_2}\big)+\varepsilon Q\big)} |R| \Ln f \Rn_{r_1, R} \Ln g\Rn_{r'_2, R},
 \label{sparse-3}
 \end{align}
 since $\epsilon>0$ is sufficiently small we have a geometrically decaying $(r_1, r'_2)$ sparse domination for $1\leq r_1\leq 2\leq r_2\leq r_{1}'$ if 
 $$-Q(1-\theta)\big(\frac{1}{r_1}-\frac{1}{r_2}\big)-\frac{\theta\beta}{2}+Q\big(\frac{1}{r_1}-\frac{1}{r_2}\big)<0\iff \frac{1}{r_1}-\frac{1}{r_2}<\frac{\beta}{2Q}.$$
 A similar argument also produces $(r_1, r'_2)$ sparse domination for the pieces $T^{l}_{j}, l\geq j\epsilon,$ in the range $1\leq r_1\leq 2\leq r_2\leq r_{1}'.$ Our proof only produces geometrically decaying sparse domination since in the dyadic scale $\mathscr{D}_{\lfloor j(1-\theta)-j\epsilon\rfloor}$ cubes are disjoint, however, to obtain a true sparse bound a similar argument can be produced as in \cite{BC}, see also \cite{LM}. Also, the operators $m(\sqrt{\La})$ are self-adjoint, therefore $(r_1, r'_2)$ sparse domination implies $(r_2, r'_1)$ sparse domination. This completes the proof of Theorem~\ref{mainthm-1} for the case $\theta>0.$ 
\end{proof}

\begin{remark}
Let $\theta<0$ and $m\in \mathscr{M}(\theta, \beta).$ The case $\theta<0$ represents low frequencies, hence we need to decompose the multiplier as done in \cite{BC, CW}. Therefore, $$m(\lambda) = \sum_{j: j \le 0} m_j (\lambda),$$
where $m_{j}=m(\lambda)\phi(\nu^{-j}\lambda).$ We can rewrite the above as $m(\lambda) = \sum_{k: k \geq 0} m_{-k} (\lambda), $ where $m_{-k}=m(\lambda)\phi(\nu^{k}\lambda)$ for $k\geq 0.$ Also $m^{-k}(\lambda):=m(\nu^{-k}\lambda)\phi(\lambda).$ Then the facts that $\|m^{-k}\|_{L^{\infty}}\leq C\nu^{k\theta\beta/2},$ and $\|m^{-k}\|_{L^2_{s}}\leq C \nu^{-k\theta(2s-\beta)/2}$ for all $k\geq 0$ and for $s\in \mathbb N,$ yield the following estimate as in Lemma~\ref{l2l2} by choosing $s\gg\lfloor-\frac{\theta\beta}{2\epsilon}\rfloor$
\begin{align*}
&\|T^l_{k}f\|_{2}\lesssim_{\epsilon} \nu^{Q(\frac{\theta\beta}{2}-Q)(k+l)}\|f\|_{2}\ \ \text{for} \ \ l>k\epsilon,\\ 
 \text{also,}\\
 &\|T_{k}f\|_{2}\lesssim\nu^{\frac{k\theta\beta}{2}}\ \ \text{for} \ \ k\geq 0.
\end{align*}
Now one can modify Lemma~\ref{l1l1}, and Lemma~\ref{linf} appropriately to obtain similar results in this case.
\end{remark}
\section{Applications}
\label{Appl}
\subsection{Quantitative estimates} In this subsection we obtain several weighted estimates for oscillating multipliers $m(\sqrt{\La}).$
Recall the following notion of Muckenhoupt weights on homogeneous spaces from \cite{BFP}. Let $1<p<\infty,$ $\omega\in A_p$ if
\begin{align}
[\omega]_{A_p}:=\sup_{B}\left(\frac{1}{|B|}\int_{B} \omega \right)\left(\frac{1}{|B|} \int_{B}\omega^{1-p'}\right)^{p-1}<\infty.
\end{align}
Also, we say $\omega\in RH_{q}$ for $1<q<\infty,$ if $[\omega]_{RH_{q}}:=\sup\limits_{B}\Ln\omega \Rn_{1, B}^{-1}\Ln  \omega \Rn_{q, B}<\infty.$ Corresponding to a sparse family $\mathcal{S},$ and $1\leq r_1, r_2\leq \infty,$ let $\Lambda_{\mathcal{S}, r_1, r'_{2}}$ denote the following bilinear form $$\Lambda_{\mathcal{S}, r_1, r'_{2}}(f, g):=\sum_{R\in \mathcal{S}}|R| \Ln f \Rn_{r_1, R} \Ln g \Rn_{r'_2, R}.$$ 
The following quantitative estimate was proved in \cite{BFP}. 
\begin{lemma}[\cite{BFP}]
\label{BFP}
For any $r_1<p<r_2,$ and $\omega\in A_{p/r_1}\cap RH_{(r_2/p)'},$ we have
$$\Lambda_{\mathcal{S}, r_1, r'_{2}}(f, g)\lesssim_{r_1, r_2, p, \mathcal{S}} \big([\omega]_{A_{p/r_1}}[\omega]_{RH_{(r_2/p)'}}\big)^{max\{\frac{1}{p-r_1}, \frac{r_2-1}{r_2-p}\}}\|f\|_{L^p(w)}\|g\|_{L^{p'}(w^{1-p'})}.$$
\end{lemma}
Motivated by \cite{BC}, as a consequence of Theorem~\ref{mainthm-1}, we now prove Theorem~\ref{quantitative} concerning weighted estimates for $m(\sqrt{\La}).$
\begin{proof}[Proof of Theorem~\ref{quantitative}] Let us first prove part $(i)$ and $(ii).$ Assume $m\in \mathscr{M}(\theta, \beta)$ with $\theta\in \mathbb{R}\setminus\{0\}$ with $Q\leq\beta\leq 2Q.$ The proof follows from Theorem~\ref{mainthm-1} and reverse H\"older's property of $A_p$ weights, see \cite{HPR}. It is easy to observe from Theorem~\ref{mainthm-1} that we have $(r_1, 1)$ sparse domination for all $r_1$ such that $0<r_1<\frac{1}{p_{\beta}}.$ Let $p_{\beta}<p<\infty,$ and $\omega\in A_{p/p_{\beta}}.$ We can always choose $\varepsilon>0$ such that $\frac{1}{p}<\frac{1}{p_{\beta}}-\frac{\varepsilon}{p}.$ Denote $\frac{1}{r}:=\frac{1}{p_{\beta}}-\frac{\varepsilon}{p}.$ Moreover, reverse H\"older's inequality ensures that the quantity $\varepsilon$ can be chosen such that $\omega\in A_{\frac{p}{r}}.$ Theorem~\ref{mainthm-1} and Lemma~\ref{BFP} imply that for any compactly supported $f$ and $g,$ there exists a sparse family $\mathcal{S}$ such that $$\big|\Ln m(\sqrt{\La})f, g\Rn\big|\lesssim \Lambda_{\mathcal{S}, r, 1}(f, g)\lesssim C([\omega]_{A_{p/p_{\beta}}}) \|f\|_{L^p(w)}\|g\|_{L^{p'}(w^{1-p'})}.$$
 Now duality concludes the proof.
\medskip

Let us now prove part $(iii).$  Let $m\in \mathscr{M}(\theta, \beta)$ with $0<\beta<Q.$  Theorem~\ref{mainthm-1} implies that we have a $(2, s')$ sparse domination for all $s'$ such that $\frac{1}{2}\leq\frac{1}{s'}<\frac{1}{2}+\frac{\beta}{2Q}.$ Let $2<p<s_{\beta},$ and $\omega\in A_{p/2}\cap RH_{(s_{\beta}/p)'}.$ By self-improving property of reverse H\"older's classes, $\omega\in RH_{(s_{\beta}/p)'(1+\delta)},$ for sufficiently small $\delta>0.$ It is easy to choose $s$ such that $2<p<s<s_{\beta}$ satisfying $\frac{1}{2}<\frac{1}{s'}<\frac{1}{2}+\frac{\beta}{2Q}=\frac{1}{s'_{\beta}}$ and $\omega\in RH_{(s/p)'}$ simultaneously. Therefore, 
 $$\big|\Ln m(\sqrt{\La})f, g\Rn\big|\lesssim \Lambda_{\mathcal{S}, 2, s'}(f, g)\lesssim C([\omega]_{A_{p/2}}, [\omega]_{RH_{(s_{\beta}/p)'}}) \|f\|_{L^p(w)}\|g\|_{L^{p'}(w^{1-p'})}.$$
Now the proof follows from duality. 
\end{proof}

\subsection{Riesz means}
For $k, \alpha, t>0,$ we define the Riesz means
\begin{align}
I_{k, \alpha, t}(\La):=k t^{-k}\int_{0}^{t}(t-s)^{k-1}e^{i s (\sqrt{\La})^{\alpha}}\, ds. 
\end{align}
Without loss of generality, let us assume that $t=1$ and simply denote $I_{k, \alpha, 1}(\La)$ by $I_{k, \alpha}(\La).$ It is well-known that the operator $I_{k, \alpha}(\La)$ can be written as $\sigma((\sqrt{\La})^{\alpha}),$ where the spectral multiplier $\sigma$ can be decomposed as $\sigma(\lambda)=c_{k}\psi(\lambda) \lambda^{-k} e^{i\lambda}+\sigma_{1}(\lambda),$ where $\sigma_{1}$ is a smooth function satisfying the Mikhlin--H\"ormander condition, and $\psi$ is a $C^{\infty}$ function such that $\psi=0$ if $0\leq \lambda\leq 1$ and $\psi\equiv 1$ for $\lambda\geq 2.$ We refer to \cite{Alex, Bui1,M, M1} and references therein. As $\sigma_{1}((\sqrt{\La})^\alpha)$ always satisfy $(1, 1)$ sparse domination, the following sparse domination follows from Corollary~\ref{mainthm-2}
\begin{align*}
 |\langle I_{k,\alpha}(\La)f, g \rangle |\lesssim_{k, \alpha, r_{1}, r_{2}}\sum_{R\in \mathcal{S}} |R|\langle f \rangle_{r_1, R}\Ln g\Rn_{r_2', R}\\
 \text{and}\,\, |\langle I_{k,\alpha}(\La)f, g \rangle |\lesssim_{k, \alpha, r_{1}, r_{2}}\sum_{R\in \mathcal{S'}} |R|\langle f \rangle_{r_2', R}\Ln g\Rn_{r_1, R}, 
 \end{align*}
 where $r_1, r_2$ satisfy 
 \begin{align*}
 \ \ \ &\left(\frac{1}{r_1}-\frac{1}{2}\right)<\frac{k}{Q},\ \ \ 1\leq r_1\leq r_2\leq 2,\ \ \ \text{or}\ \ \ &\left(\frac{1}{r_1}-\frac{1}{r_2}\right)<\frac{k}{Q},\ \ \ 1\leq r_1\leq 2\leq r_2\leq r'_{1}.
 \end{align*}    
The above sparse domination and Theorem~\ref{quantitative} yield the following weighted estimates.
\begin{enumerate}[i)]
    \item 
Let $k\geq Q.$ Then $I_{k, \alpha}({\La})$ maps $L^p(\omega)$ to $L^p(\omega)$ for all $1<p<\infty$ and $\omega\in A_{p}.$ 
\item Let $\frac{Q}{2}\leq k<Q.$ Then $I_{k, \alpha}({\La})$ maps $L^p(\omega)$ to $L^p(\omega)$ for $p_{k}<p<\infty$ and $\omega\in A_{p/p_{k}},$ where $p_{k}:=\frac{Q}{k}.$
\item Let $0<k<\frac{Q}{2}.$ Then $I_{k, \alpha}({\La}): L^p(\omega)\to L^p(\omega)$ for all $2<p<s_{k},\ \  \omega\in A_{p/2}\cap RH_{(s_{k}/p)'},$ where $\frac{1}{s_{k}}:=\frac{1}{2}-\frac{k}{Q}.$
\end{enumerate}

\subsection{Dispersive equations} Let $f\in C^{\infty}_{0}(G)$ and $\alpha\in \mathbb N.$ Consider the dispersive equation 
\begin{align*}
 i\, \partial_{t} u+ (\sqrt{\La})^{\alpha}\ u=0,\,\, u(\cdot, 0)=f.    
\end{align*}
Then $u(x, t)=e^{it (\sqrt{\La})^{\alpha}}f(x, t).$ For a fixed time $t,$ rescaling the operator $\sqrt{\La}$ by $t^{1/\alpha}\sqrt{\La},$ one can prove the following as a consequence of Corollary~\ref{mainthm-2}
\begin{align*}
|\langle u(\cdot, t) , g \rangle|\lesssim_{\beta, \alpha, r_1, r_2, t} \sum_{R\in \mathcal{S}} |R|\langle (I+\sqrt{\La})^{\beta}f \rangle_{r_1, R}\Ln g\Rn_{r_2', R}\ \ 
\end{align*}
whenever $\left(\frac{1}{r_1}-\frac{1}{2}\right)<\frac{\beta}{\alpha Q}, 1\leq r_1\leq r_2\leq 2$ or $\left(\frac{1}{r_1}-\frac{1}{r_2}\right)<\frac{\beta}{\alpha Q}, 1\leq r_1\leq 2\leq r_2\leq r'_{1}.$ Let $W^{s, p}_{\omega}$ denotes the non-homogeneous weighted Sobolev space $W^{s, p}_{\omega}=\{f: \|f\|_{W^{s, p}_{\omega}}:=\|(I+\sqrt{\La})^{s}f\|_{L^p(\omega)}<\infty\}.$
As an application of Theorem~\ref{quantitative}, we can derive the following weighted estimates:
\begin{enumerate}[i)]
    \item 
 Let $1<p<\infty$ and $\omega\in A_{p}.$ Then $\|u(\cdot, t)\|_{L^p(\omega)}\lesssim \|f\|_{W^{\beta, p}_{\omega}}$ provided $\beta\geq \alpha\, Q.$   
 \item
 Let $\frac{\alpha\,Q}{2}\leq \beta<\alpha Q.$ Then $\|u(\cdot, t)\|_{L^p(\omega)}\lesssim \|f\|_{W^{\beta, p}_{\omega}}$ holds for all $p_{\alpha, \beta}<p<\infty$ and $\omega\in A_{p/p_{\alpha, \beta}},$ where $p_{\alpha, \beta}:=\frac{Q\,\alpha}{\beta}.$ 
 \item
 Finally, let $0<\beta<\frac{\alpha\,Q}{2}.$ We also have that $\|u(\cdot, t)\|_{L^p(\omega)}\lesssim \|f\|_{W^{\beta, p}_{\omega}}$ holds for all $2<p<s_{\alpha, \beta},\ \  \omega\in A_{p/2}\cap RH_{(s_{\alpha, \beta}/p)'},$ where $\frac{1}{s_{\alpha, \beta}}:=\frac{1}{2}-\frac{\beta}{\alpha Q}.$
\end{enumerate}

\end{document}